\documentclass[a4paper,draft]{amsproc}
\usepackage[all,cmtip]{xy}

\usepackage{tikz-cd}
\usepackage{amsmath}
\usepackage{amssymb}
\usepackage{tikz-cd}
\usepackage[hyphens]{url} 

\usepackage[utf8]{inputenc}  
\DeclareUnicodeCharacter{202F}{ }  

\theoremstyle{plain}
 \newtheorem{theorem}{Theorem}[section]
 \newtheorem{proposition}{Proposition}[section]

\theoremstyle{definition}
 \newtheorem{example}{Example}[section]
 \newtheorem{definition}{Definition}[section]
\theoremstyle{remark}
 \newtheorem{remark}{Remark}[section]
 \numberwithin{equation}{section}

\renewcommand{\leq}{\leqslant}

\title[The ideal structure of the minimal Tensor Product of TROs]{The ideal structure of the minimal Tensor Product of Ternary Rings of Operators}

\author[Arpit Kansal and Vandana Rajpal]{\bfseries Arpit Kansal and Vandana Rajpal} 

\DeclareUnicodeCharacter{202F}{\,}

\begin{document}

\begin{flushleft}\baselineskip9pt\scriptsize

\end{flushleft}
\vspace{18mm} \setcounter{page}{1} \thispagestyle{empty}

\begin{abstract}
Let \( V \) be a ternary ring of operator and \( B \) a \( C^* \)-algebra. We study the structure of the ideal space of the operator space injective tensor product \( V \otimes^{\mathrm{tmin}} B \) via two maps:
\[
\Phi(I, J) = \ker(q_I \otimes^{\mathrm{tmin}} q_J) \quad \text{and} \quad \Delta(I, J) = I \otimes^{\mathrm{tmin}} B + V \otimes^{\mathrm{tmin}} J.
\]
We show that \( \Phi \) is continuous with respect to the hull-kernel topology, and that its restriction to primitive and prime ideals defines a homeomorphism onto dense subsets of the respective ideal spaces of \( V \otimes^{\mathrm{tmin}} B \). 

We prove that if \( \Phi = \Delta \), then \( \Phi \) induces a homeomorphism between the space of minimal primal ideals of \( V \otimes^{\mathrm{tmin}} B \) and the product of the spaces of minimal primal ideals of \( V \) and \( B \). 
\end{abstract}

\maketitle

\section{\textbf{Introduction}}

The study of ideal structures in tensor products of operator algebras has been a central theme in $C^{\ast}$-algebra theory. Of particular interest is understanding how the ideal spaces of the individual algebras relate to the ideal space of their tensor product. This line of investigation dates back to seminal work on minimal tensor products and has since been extended to various tensor norms and categories of algebras. In this context, classical results by Tomiyama, Archbold, Lazar, Somerset, and others have characterized the primitive and prime ideal spaces of $A\otimes^{\text{min}} B$
(see, e.g., \cite{ Lazer, za, walk, Kanimin}). 

Motivated by these developments, we study an analogous problem in the category of ternary rings of operators (TROs), which generalize 
 $C^{\ast}$-algebras by focusing on ternary operations rather than binary multiplication. TROs arise naturally as off-diagonal corners of $C^{\ast}$-algebras and play a key role in the theory of Hilbert
 $C^{\ast}$-modules, Morita equivalence, and non-self-adjoint operator algebras. While the ideal theory of $C^{\ast}$-algebras is well-developed, TROs present new challenges due to the absence of a multiplicative structure, requiring more refined techniques to study their ideal lattices and dual spaces.

 In this paper, we focus on the ideal structure of the t-minimal tensor product \( V \otimes^{\mathrm{tmin}} B \). This tensor product, defined as the closure of the algebraic tensor product \( V \otimes B \) under a natural ternary product inherited from their embeddings into operator spaces, generalizes the minimal \( C^* \)-tensor product and preserves the TRO structure. Despite its natural construction, the behavior of ideals under \( \otimes^{\mathrm{tmin}} \) has not been thoroughly investigated. The goal of this paper is to initiate a systematic study of this ideal structure and to understand the extent to which known results from the \( C^* \)-algebraic setting carry over to this more general context.

We introduce and analyze the canonical map
\[
\Phi : \mathrm{Id}'(V) \times \mathrm{Id}'(B) \to \mathrm{Id}'(V \otimes^{\mathrm{tmin}} B), \quad \Phi(I_1, I_2) := \ker(q_{I_1} \otimes^{\mathrm{tmin}} q_{I_2}),
\]
where \( q_{I_1} : V \to V/I_1 \) and \( q_{I_2} : B \to B/I_2 \) denote the respective quotient maps, and \( \mathrm{Id}'(X) \) denotes the collection of all proper closed ideals of a space \( X \). We examine the algebraic and topological features of this map.
 Our study focuses particularly on how classes of ideals such as prime, primitive and primal ideals are reflected and preserved under this map. We provide a complete characterization of the image of \( \Phi \) in the case of primal ideals of \( V \otimes^{\mathrm{tmin}} B \) in terms of the primal ideals of \( V \) and \( B \). These results extend and generalize known descriptions from the \( C^* \)-algebraic framework to the setting of TROs.

The paper is organized as follows. In Section~2, we collect the essential background and preliminary results on TROs and the construction of the t-minimal tensor product with a $C^*$-algebra. We also discuss the associated linking $C^*$-algebra $\mathcal{A}(V)$, the correspondence between ideals of $V$ and $\mathcal{A}(V)$, and review various classes of ideals such as prime, primitive and primal ideals, along with their topological properties.

In Section~3, we introduce and study the maps $\Phi$ and $\Delta$ between the ideal spaces of $V$, $B$, and $V \otimes^{\text{tmin}} B$. Our primary focus is on the algebraic and topological behavior of $\Phi$, including its continuity and interaction with different classes of ideals. In particular, we show that $\Phi$ restricts to homeomorphisms onto dense subsets of the prime, primitive ideal spaces, and becomes a homeomorphism onto its image when extended to all proper closed ideals.

In Section~4, we study the structure of primal ideals in the t-minimal tensor product $V \otimes^{\mathrm{tmin}} B$. We begin by recalling the definition and basic properties of primal ideals, followed by a detailed analysis of how primality behaves under the canonical map $\Phi$ and its interaction with the product-type map $\Delta$. In particular, we characterize the image of $\Phi$ restricted to the set of primal ideals of $V \otimes^{\mathrm{tmin}} B$ and establish the necessary and sufficient conditions for an ideal of the form $\Delta(I, J)$ to be primal. We also show that, under natural assumptions,  the map
\[
\Phi \colon \operatorname{Min\text{-}Primal}(V) \times \operatorname{Min\text{-}Primal}(B) \to \operatorname{Min\text{-}Primal}(V \otimes^{\mathrm{tmin}} B)
\]
is a homeomorphism.

.

\section{Preliminary results} 

Let \( V \subseteq B(\mathcal{H}, \mathcal{K}) \) be a TRO, that is, a norm-closed subspace closed under the ternary product:
\[
(x, y, z) \mapsto x y^* z, \quad \text{for all } x, y, z \in V.
\]
TROs naturally arise as off-diagonal corners in certain $C^*$-algebras. In particular, following \cite{MHA}, to every TRO \( V \) one can associate a linking $C^*$-algebra \( \mathcal{A}(V) \) given by:
\[
\mathcal{A}(V) =
\begin{bmatrix}
C(V) & V \\
V^* & D(V)
\end{bmatrix},
\]
where \( C(V) = \mathrm{C}^*(VV^*) \) and \( D(V) = \mathrm{C}^*(V^*V) \) are the $C^*$-algebras generated by \( VV^* \subseteq B(\mathcal{K}) \) and \( V^*V \subseteq B(\mathcal{H}) \), respectively.

\begin{definition}
Let \( V \) be a TRO, and let \( I \subseteq V \) be a closed subspace. We consider several classes of ideals within \( V \) as follows:
\begin{enumerate}
    \item A subspace \( I \) is called an ideal of \( V \) if it satisfies the inclusion
    \[
    IV^*V + VV^*I \subseteq I.
    \]
    The set of all closed ideals of \( V \) is denoted by \( \operatorname{Id}(V) \). We equip \( \operatorname{Id}(V) \) with the \emph{weak topology} \( \tau_w \), which is generated by sub-basis sets of the form
    \[
    U(J) = \{ I \in \operatorname{Id}(V) : I \not\supseteq J \}, \quad \text{for } J \in \operatorname{Id}(V).
    \]
    
    \item An ideal \( P \in \operatorname{Id}(V) \) is said to be prime if for any ideals \( I, J, K \in \operatorname{Id}(V) \), the inclusion \( IJK \subseteq P \) implies that at least one of \( I \subseteq P \), \( J \subseteq P \), or \( K \subseteq P \) holds. The collection of prime ideals of \( V \) is denoted by \( \operatorname{Prime}(V) \), equipped with the subspace topology inherited from \( \operatorname{Id}(V) \).
    
    \item An ideal \( I \in \operatorname{Id}(V) \) is called primitive if it is the kernel of an irreducible representation of \( V \). The set of primitive ideals is denoted by \( \operatorname{Prim}(V) \), also equipped with the subspace topology inherited from \( \operatorname{Id}(V) \).
    
\end{enumerate}
\end{definition}

\vspace{0.5em}

We now recall a natural map that connects the ideal structure of a TRO with that of its linking \( C^* \)-algebra. The canonical map
\[
\theta : \operatorname{Id}(V) \to \operatorname{Id}(\mathcal{A}(V))
\]
is defined by
\[
\theta(I) := \mathcal{A}(I),
\]
where \( \mathcal{A}(I) \) denotes the ideal of the linking $C^{\ast}$-algebra \( \mathcal{A}(V) \) generated by \( I \). The following theorem summarizes known results regarding the behavior of \( \theta \) on various ideal classes.

\begin{theorem} \label{sec3main1}
The map \( \theta : \operatorname{Id}(V) \to \operatorname{Id}(\mathcal{A}(V)) \) satisfies the following properties:
\begin{enumerate}
    \item[(i)] \textup{(\cite{AAV}, Theorem 2.6)} The restriction of \( \theta \) to \( \operatorname{Prime}(V) \) is a homeomorphism onto \( \operatorname{Prime}(\mathcal{A}(V)) \).
    
    \item[(ii)] \textup{(\cite{AA}, Proposition 3)} The restriction of \( \theta \) to \( \operatorname{Prim}(V) \) is a homeomorphism onto \( \operatorname{Prim}(\mathcal{A}(V)) \).
    
    
    
\end{enumerate}
\end{theorem}

For each subset \( S \) of \( V \), define  
\[
\operatorname{hull}(S) = \{ P \in \operatorname{Prim}(V) \mid S \subseteq P \}
\]  
and for each subset \( B \) of \( \operatorname{Prim}(V) \), let  
\[
\operatorname{ker}(B) = \bigcap \{ P \mid P \in B \}.
\]  
For any ideal \( I \) of \( V \), we have \( \operatorname{ker}(\operatorname{hull}(I)) = I \). As a consequence, we obtain 
\[
U(J) = \{ I \in \operatorname{Id}(V) \mid \operatorname{hull}(I) \cap V(J) \neq \emptyset \},
\]  
where \( V(J) = \{ P \in \operatorname{Prim}(V) \mid P \not\supseteq J \} \).

Let \( V \subseteq B(\mathcal{H}_1, \mathcal{K}_1) \) and \( W \subseteq B(\mathcal{H}_2, \mathcal{K}_2) \) be two TROs. On the algebraic tensor product \( V \otimes W \), one defines a natural triple product:
\[
(v_1 \otimes w_1)(v_2 \otimes w_2)^*(v_3 \otimes w_3) = (v_1 v_2^* v_3) \otimes (w_1 w_2^* w_3),
\]
for all \( v_i \in V \), \( w_i \in W \). 

The t-minimal tensor product \( V \otimes^{\mathrm{tmin}} W \) is then defined as the closure of this algebraic tensor product under the induced operator norm in \( B(\mathcal{H}_1 \otimes \mathcal{H}_2, \mathcal{K}_1 \otimes \mathcal{K}_2) \). That is, \( V \otimes^{\mathrm{tmin}} W \subseteq B(\mathcal{H}_1 \otimes \mathcal{H}_2, \mathcal{K}_1 \otimes \mathcal{K}_2) \) is again a TRO with respect to this ternary product. For more details, the reader is referred to (\cite{ER}, \cite{MK_JR}). 

An important feature of this construction is that when both \( V \) and \( W \) are $C^*$-algebras, the t-minimal tensor product agrees with the minimal $C^*$-tensor product:
\[
V \otimes^{\mathrm{tmin}} W = V \otimes^{\min} W.
\]

Moreover, $\mathcal{A}(V \otimes^{\mathrm{tmin}} B)= \mathcal{A}(V)\otimes^{\min} B $ (\cite{MK_JR}, Proposition 3.1). The t-minimal tensor product also behaves well under homomorphisms, as captured in the following proposition:

\begin{proposition}[\cite{ER}, Proposition 9.2.5]
Let \( V_i, W_i \) be TROs for \( i = 1, 2 \), and let \( f_i : V_i \to W_i \) be maps. If \( f_1 \) and \( f_2 \) are TRO-homomorphisms, then so is the map \( f_1 \otimes f_2 : V_1 \otimes^{\mathrm{tmin}} V_2 \to W_1 \otimes^{\mathrm{tmin}} W_2 \).

\end{proposition}

\section{Ideals of $V \otimes^{tmin} B$}

In this section, we investigate the topological behavior of the canonical map
\[
\Phi : \mathrm{Id}'(V) \times \mathrm{Id}'(B) \to \mathrm{Id}'(V \otimes^{\mathrm{tmin}} B), \quad \Phi(I_1, I_2) := \ker(q_{I_1} \otimes^{\mathrm{tmin}} q_{I_2}),
\]
Alongside, we also consider the map
\[
\Delta(I_1, I_2) := I_1 \otimes^{\mathrm{tmin}} B + V \otimes^{\mathrm{tmin}} I_2.
\]

Our aim is to explore various topological aspects of $\Phi$, particularly when restricted to distinguished classes of ideals such as primitive and prime ideals. We show that under appropriate conditions, these restrictions are continuous and yield homeomorphisms onto their respective images. The main results of this section are summarized in the following theorem.

\begin{theorem} \label{thm3main}
Let $V$ be a TRO and $B$ a $C^*$-algebra, then the following hold:
\begin{enumerate}
    \item[(i)] The restriction of $\Phi$  to  $\mathrm{Prim}(V) \times \mathrm{Prim}(B)$ is a homeomorphism onto its image which is dense in $\mathrm{Prim}(V\otimes^{tmin} B)$. If $V$ or $B$ is exact, then $\Phi$ maps $\mathrm{Prim}(V) \times \mathrm{Prim}(B)$ onto $\mathrm{Prim}(V\otimes^{tmin} B)$.
     \item[(ii)] If $I, K \in Id'(V)$ and $J,L \in Id'(B)$ are such that $I \subseteq K$ and $J \subseteq L$  then $\Phi(I,J) \subseteq \Phi(K,L) $.
    \item[(iii)] The mapping $\Phi$ is $\tau_w$-continuous.
    \item[(iv)] The restriction of $\Phi$ to $\mathrm{Id}'(V) \times \mathrm{Id}'(B)$ is a homeomorphism onto its image in $\mathrm{Id}'(V\otimes^{tmin} B)$.
    \item[(v)] The restriction of $\Phi$ to $\mathrm{Prime}(V) \times \mathrm{Prime}(B)$ is a homeomorphism onto 
a dense subset of $\mathrm{Prime}(V\otimes^{tmin} B)$.

\end{enumerate}
\end{theorem}

\begin{proof}
    (i)  We first note that the map  $\Phi$ restricted  to  $\mathrm{Prim}(V) \times \mathrm{Prim}(B)$ is well defined by (\cite{wulf}, \S \: 2), (\cite{AAV}, Theorem 2.6) and the property that $\ker (q_{\mathcal{A}(I_1)} \otimes^{min} q_{I_2} )=\mathcal{A}(\ker (q_{I_1} \otimes^{tmin} q_{I_2} ))$ (\cite{AAV}, Lemma $2.7$).
The map $\theta \times 1:\operatorname{Prim}(V) \times \operatorname{Prim}(B) \to \operatorname{Prim}(\mathcal{A}(V)) \times \operatorname{Prim}(B) $ is a homeomorphism. Moreover, the map $\Psi: \operatorname{Prim}(V \otimes^{\text{tmin}} B) \to \operatorname{Prim}(\mathcal{A}(V) \otimes^{\text{min}} B)$ is also a homeomorphism (\cite{MK_JR}, Proposition $3.1$). Furthermore, the map $\Phi': \operatorname{Prim}(\mathcal{A}(V)) \times \operatorname{Prim}(B) \to \operatorname{Prim}(\mathcal{A}(V) \otimes^{\text{min}} B)  $ is a homeomorphism onto its image (\cite{Gui}, Theorem $5$). Thus, we obtain the following commutative diagram

\[
\begin{tikzcd}
  \operatorname{Prim}(V) \times \operatorname{Prim}(B)
    \arrow[r, "\Phi"]
    \arrow[d, swap, "\theta \times 1"]
  &
  \operatorname{Prim}(V \otimes^{\text{tmin}} B)
    \arrow[d, swap, "\Psi"]
  \\
  \operatorname{Prim}(\mathcal{A}(V)) \times \operatorname{Prim}(B)
    \arrow[r, swap, "\Phi'"]
  &
  \operatorname{Prim}(\mathcal{A}(V)\otimes^{\mathrm{min}} B)
\end{tikzcd}
\]
Using the commutativity of the diagram, it follows that the restriction of $\Phi$ to \newline $\operatorname{Prim}(V) \times \operatorname{Prim}(B)$  is a homeomorphism onto its image.  
If $V$ or $B$ is exact, then  $\mathcal{A}(V)$ is exact by (\cite{AA}, Lemma 3). Thus the result follows from  (\cite{EB}, Proposition 2.16 and 2.17).\\

(ii) Let $I \subseteq K$ be ideals in $V$ and $J \subseteq L$ in $B$. Then there exists a natural *-homomorphism
\[
\theta_{K, L} : (V / I) \otimes^{\ tmin} (B / J) \to (V / K) \otimes^{\ tmin} (B / L)
\]
which is defined on the algebraic tensor product by
\[
\theta_{K, L}((v + I) \otimes (b + J)) = (v + K) \otimes (b + L),
\]
and extends by continuity to the minimal tensor product. (Proposition $2.1$). This homomorphism satisfies:
\[
\theta_{K, L} \circ (q_{I} \otimes^{tmin} q_{J}) = q_{K} \otimes^{tmin} q_{L},
\]
from which it follows that:
\[
\ker(q_{I} \otimes^{tmin} q_{J}) \subseteq \ker(q_{K} \otimes^{tmin} q_{L}),
\]
i.e.,
\[
\Phi(I, I) \subseteq \Phi(K, L).\]
(iii)  We have  to show that if $K \in \mathrm{Id}'(V\otimes^{tmin} B)$ and $I_0 \in \operatorname{Id}(V)$, $J_0 \in \operatorname{Id}(B)$ satisfy $\Phi(I_0, J_0) \in U(K)$,  there exist $\tau_w$-neighbourhoods $U_1$ of $I_0$ and $U_2$ of $J_0$ in $\operatorname{Id}(V)$ and $\operatorname{Id}(B)$, respectively, such that $\Phi(U_! \times U_2) \subseteq U(K)$. 

 Since $I_0$ and $J_0$ are necessarily proper, it follows, by (\cite{Lazer}, Remark $2.4$) and (\cite{AAV}, Theorem $2.6(4))$ that the kernel of this map can be expressed as  
\[
\ker(q_{\mathcal{A}(I_0)} \otimes^{\mathrm{min}} q_{J_0}) = \bigcap_{(P, Q) \in \operatorname{hull}(I_0) \times \operatorname{hull}(J_0)} \ker(q_{\mathcal{A}(P)} \otimes^{\mathrm{min}} q_Q).
\]  

On the other hand, it follows from (\cite{AAV}, Lemma $2.7$) that  
\[
\ker(q_{\mathcal{A}(I_0)} \otimes^{\mathrm{min}} q_{J_0}) = \mathcal{A}(\ker(q_{I_0} \otimes^{\mathrm{tmin}} q_{J_0})),
\]  
and since the functor \(\mathcal{A}\) is injective on ideals, we conclude that  
\[
\ker(q_{I_0} \otimes^{\mathrm{tmin}} q_{J_0}) = \bigcap_{(P, Q) \in \operatorname{hull}(I_0) \times \operatorname{hull}(J_0)} \ker(q_P \otimes^{\mathrm{tmin}} q_Q).
\]
 
That is, $\Phi(I_0, J_0)=\bigcap_{P \in hull(I_0),  Q \in hull(J_0)} \Phi(P,Q)$.
Since $\Phi$ is continuous on $\text{Prim}(V) \times \text{Prim}(B)$  and
\[
\Phi(P_0, Q_0) \in U(K) \cap \text{Prim}(V \otimes^{tmin} B),
\]
there exist open neighbourhoods $V_0$ of $P_0$ in $\text{Prim}(V)$ and $W_0$ of $Q_0$ in $\text{Prim}(B)$ such that $\Phi(V_0 \times W_0) \subseteq U(K)$.  Now, define
\[
U_1 = \{I \in \operatorname{Id}(V): hull(I) \cap V_0 \neq \emptyset\}, \quad U_2 = \{J \in \operatorname{Id}(B): hull(J) \cap W_0 \neq \emptyset\}.
\]
Then $I_0 \in U_1$ and $J_0 \in U_2$, and $U_1$,  $U_2$ are $\tau_w$-open in $\operatorname{Id}(V)$ and  $\operatorname{Id}(B)$, respectively. Finally, let $I \in U_1$ and $J \in U_2$. Then, there exist $P \in hull(I) \cap V_0$ and $Q \in hull(J) \cap W_0$.  Then by (ii) part, $\Phi(P,Q) \supseteq \Phi(I,J)$. Since $\Phi(P,Q)\in U(K) $ and so $\Phi(I,J)\in U(K)$.

(iv) Consider the map $\Psi: \mathrm{Id}'(V\otimes^{tmin} B) \to \mathrm{Id}'(V) \times \mathrm{Id}'(B) $ defined as $\Psi(I)=(I_V, I_B)$, $I_V=\{ v\in V: v\otimes b\in I, for\: all \: b\in B\}$, and $I_B=\{ b\in B: v\otimes b\in I, for\: all \:v\in V\}$. We first show that this map is well defined. Let $I$ be the closed ideal in  $V\otimes^{tmin} B$, then by the construction of the associated functor $\mathcal{A}$, $\mathcal{A}(I)$ is a closed ideal in
$\mathcal{A}(V)\otimes^{min} B$ and therefore the corresponding ideal in $\mathcal{A}(V)$ is given by \[
\mathcal{A}(I)_{\mathcal{A}(V)} = \mathcal{A}(I_V) = \{ a \in \mathcal{A}(V) \mid a \otimes B \subseteq \mathcal{A}(I) \}
\]
 This implies that $\mathcal{A}(I_V)$   is a closed ideal of $\mathcal{A}(V)$ \cite{EB} and hence $I_V$  is a closed ideal of $V$. The same happens with $I_B$.

Next, we show that the composition $\Psi \circ \Phi$ restricted to $\mathrm{Id}'(V) \times \mathrm{Id}'(B)$ is the identity map.

Let $(I_1, I_2) \in \mathrm{Id}'(V) \times \mathrm{Id}'(B)$ and set $I := \Phi(I_1, I_2) = \ker(q_{I_1} \otimes^{tmin} q_{I_2})$. Clearly, $I_1 \subseteq I_V$ and $I_2 \subseteq I_B$. Now let $v \in I_V$. Then $v \otimes b \in I= \ker(q_{I_1} \otimes q_{I_2}) $, and so $(v+I_1)\otimes (b+I_2)=0$, it follows that $v\in I_1$. Hence, $I_V = I_1$. A similar argument shows $I_B = I_2$, proving that $\Psi \circ \Phi = \mathrm{id}$.  We now show that $\Psi$ is a  continuous map. Consider a basic open set $U(I_1) \times U(I_2)$. Then $\Psi^{-1}(U(I_1) \times U(I_2))=\{I \in \operatorname{Id}(V \otimes^{tmin} B) \mid \operatorname{hull} (I_{V}))\cap V(I_1)\neq \emptyset, \operatorname{hull} (I_{B}) \cap V(I_2)\neq \emptyset\}=\{I \in \operatorname{Id}(V \otimes^{tmin} B) \mid  I_{V} \not\supseteq I_1,   I_{B} \not\supseteq I_2\}=\{I \in \operatorname{Id}(V \otimes^{tmin} B) \mid I \supseteq  V \otimes^{tmin}  I_{B}  \not\supseteq  V \otimes^{tmin}  I_{2}, I \supseteq  I_V \otimes^{tmin} B  \not\supseteq  I_1 \otimes^{tmin}  B\} = \{I \in \operatorname{Id}(V \otimes^{tmin} B) \mid I \not\supseteq  V \otimes^{tmin} I_{2}, I   \not\supseteq  I_1 \otimes^{tmin}  B\}= U(V \otimes^ {tmin}  I_{2}) \cap  U(I_1 \otimes^{tmin}  B) $, which is open in  $V \otimes^{tmin} B$. It follows that $\Phi$ maps open subsets of  $ \mathrm{Id}'(V) \times \mathrm{Id}'(B)$ onto relatively open subsets of  $V \otimes^{tmin} B$. Thus,  $\Phi$ is a homeomorphism onto its image.

To show density: let $W_r \subseteq \mathrm{Prim}(V \otimes^{\ tmin} B)$ be open subsets for $1 \leq r \leq m$.  Since the map $\mathrm{Prim}(V) \times \mathrm{Prim}(B)$ is a homeomorphism onto its image which is dense in $\mathrm{Prim}(V\otimes^{tmin} B)$ by (i) part. There exist primitive ideals $P_r^{(1)} \in \mathrm{Prim}(V)$ and $P_r^{(2)} \in \mathrm{Prim}(B)$ such that $\Phi(P_r^{(1)}, P_r^{(2)}) \in W_r$. Define
\[
I_1 := \bigcap_{r=1}^m P_r^{(1)}, \quad I_2 := \bigcap_{r=1}^m P_r^{(2)}.
\]
Then $\Phi(I_1, I_2) \in \{ I \in \mathrm{Id}(V \otimes^{tmin}  B) \mid \mathrm{hull}(I) \subseteq W_r \text{ for all } r \}$, proving that the image of $\Phi$ is dense in $\mathrm{Id}'(V \otimes^{\ tmin} B)$.

(v) The result follows lines similar to those used for (i) using (\cite{AA}, Proposition 3) and (\cite{EB}, Lemma 2.13(v)).
\\
\end{proof}

\section{Primal Ideals of $V \otimes^{tmin} B$}

In this section, we investigate the structure  of primal and minimal primal ideals in  \( V \otimes^{\mathrm{tmin}} B \). To do so, we begin by recalling the notion of primality for ideals in a TRO.

\begin{definition}
An ideal \( P \) of a TRO \( V \) is called \emph{primal} if for every collection of ideals \( J_1, J_2, \dots, J_{2n+1} \subseteq V \), with \( n \in \mathbb{N} \), the condition
\[
J_1 J_2 \cdots J_{2n+1} = \{0\}
\]
implies that \( J_k \subseteq P \) for some \( 1 \leq k \leq 2n+1 \). The set of all primal ideals of \( V \) is denoted by \( \operatorname{Primal}(V) \).
\end{definition}

\begin{example}
Let \( X \) be a compact Hausdorff space, and let \( V \) be a TRO. Consider the TRO \( C(X, V) \), consisting of all continuous functions from \( X \) into \( V \), equipped with pointwise operations. For any \( x_0 \in X \) and any primal ideal \( P \subseteq V \), define
\[
I_{x_0, P} := \{ f \in C(X, V) : f(x_0) \in P \}.
\]
We claim that \( I_{x_0, P} \) is a primal ideal in \( C(X, V) \). Suppose not. Then there exist ideals \( J_1, \dots, J_{2n+1} \subseteq C(X, V) \) with \( J_1 J_2 \cdots J_{2n+1} = \{0\} \) and \( J_k \nsubseteq I_{x_0, P} \) for all \( k \). 

Let \( \varepsilon_{x_0} : C(X, V) \to V \) denote the evaluation map at \( x_0 \), given by \( \varepsilon_{x_0}(f) = f(x_0) \). Applying \( \varepsilon_{x_0} \), we obtain TRO ideals \( \varepsilon_{x_0}(J_k) \subseteq V \) with 
\[
\varepsilon_{x_0}(J_1) \cdots \varepsilon_{x_0}(J_{2n+1}) = \{0\}.
\]
Since \( P \) is primal, there exists \( k \) such that \( \varepsilon_{x_0}(J_k) \subseteq P \), implying \( J_k \subseteq I_{x_0, P} \), a contradiction. Hence, \( I_{x_0, P} \in \operatorname{Primal}(C(X, V)) \).
\end{example}

\begin{remark}
An ideal \( P \subseteq V \) is primal if and only if the ideal \( \mathcal{A}(P) \subseteq \mathcal{A}(V) \) is a primal ideal in the linking \( C^* \)-algebra. Indeed, this follows from the fact that every ideal in \( \mathcal{A}(V) \) is of the form \( \mathcal{A}(I) \) for some ideal \( I \subseteq V \), and that the ternary product of an odd number of ideals in \( V \) reduces to their intersection (\cite{AA}, Lemma $1$). The correspondence between products and intersections of ideals is preserved under the functor \( \mathcal{A} \), as is the inclusion \( I \subseteq P \) if and only if \( \mathcal{A}(I) \subseteq \mathcal{A}(P) \).
\end{remark}

\begin{proposition}
 For proper primal ideals \( I \subseteq V \) and \( J \subseteq B \), $\Phi(I, J)$ is a primal ideal in \( V \otimes^{\mathrm{tmin}} B \). In particular, if  $V$ or $B$ is exact, then $I \otimes^{\text{tmin}} B+ V \otimes^{\text{tmin}} J $ is a primal ideal of $V \otimes^{\text{tmin}} B.$
\end{proposition}

\begin{proof}
    As $q_I \otimes^{\text{tmin}} q_J: V \otimes^{\text{tmin}}B \to V/I \otimes^{\text{tmin}} B/J $ is a ternary homomorphism, so applying the functor $\mathcal{A}$ and using (\cite{AKKKK}, Proposition $4.6$)  $\mathcal{A}(q_I \otimes^{\text{tmin}} q_J): \mathcal{A}(V) \otimes^{\text{min}}B \to \mathcal{A}(V)/\mathcal{A}(I) \otimes^{\text{min}} B/J $ is a $C^{\ast}$-homomorphism. Since $\operatorname{ker}(\mathcal{A}(q_I \otimes^{\text{tmin}} q_J))= \mathcal{A}(\operatorname{ker}(q_I \otimes^{\text{tmin}} q_J))$  (\cite{AAV}, Lemma $2.7$) and   $\operatorname{ker}(\mathcal{A}(q_I \otimes^{\text{tmin}} q_J))$ is primal (\cite{za}, Proposition $3.3$) so $\ker(q_I \otimes^{\mathrm{tmin}} q_J)$ is a primal ideal in \( V \otimes^{\mathrm{tmin}} B \) by last remark.   Second assertion follows immediately from (\cite{AA}, Lemma $9$). 
\end{proof}

\begin{remark}
    Since $C(X)$ is exact, Example $4.1$ can also be obtained from Proposition $4.1$ by considering $C(X, V)$ as $C(X) \otimes^{\text{tmin}} V$ (\cite{AAV}, Proposition $3.4$).  
\end{remark}

The next result shows that the second part of the previous proposition also admits a converse under the same assumptions.

\begin{proposition}
If \( V \) or \( B \) is exact and $I \otimes^{\mathrm{tmin}} B + V \otimes^{\mathrm{tmin}} J$ is a primal ideal in \( V \otimes^{\mathrm{tmin}} B \), then \( I \) and \( J \) are primal ideals.
\end{proposition}

\begin{proof}
Using Remark 4.1, $\mathcal{A}(I \otimes^{\mathrm{tmin}} B + V \otimes^{\mathrm{tmin}} J) = \mathcal{A}(I) \otimes^{\text{min}} B + \mathcal{A}(V) \otimes^{\text{min}} J$ (\cite{MK_JR}, Proposition 3.1)  is a primal ideal in \( \mathcal{A}(V) \otimes^{\mathrm{min}} B \). Since \( V \) or \( B \) is exact, it follows from (\cite{MK_JR}, Theorem 4.4) that \( \mathcal{A}(V) \) or \( B \) is exact. Then by (\cite{Kanimin}, Lemma 1.2), the ideals \( \mathcal{A}(I) \) and \( J \) are primal. Consequently, \( I \) and \( J \) are primal ideals.
\end{proof}

We observe that an ideal \( I \) of a TRO \( V \) is minimal if and only if the corresponding ideal \( \mathcal{A}(I) \) is a minimal ideal of \( \mathcal{A}(V) \). Using this fact, the proof of the next result follows along the same lines as the proof of the previous proposition, combined with (\cite{Kanimin}, Lemma 1.3).

\begin{proposition}
  If \( V \) or \( B \) is exact  then \( I \) and \( J \) are minimal primal ideals if and only if $I \otimes^{\mathrm{tmin}} B + V \otimes^{\mathrm{tmin}} J$
is a minimal primal ideal in \( V \otimes^{\mathrm{tmin}} B \).
\end{proposition}

The previous result shows that under the assumption that either \( V \) or \( B \) is exact, the map \( \Phi \) defines a bijection from $\operatorname{Min\text{-}Primal}(V) \times \operatorname{Min\text{-}Primal}(B)$ onto  $\operatorname{Min\text{-}Primal}(V \otimes^{\mathrm{tmin}} B)$. In the next result, we show that \( \Phi \) is, in fact, a homeomorphism under a condition that is weaker than exactness.

\begin{theorem}
Suppose that \( V \) is a TRO and \( B \) is a \( C^* \)-algebra such that $\Phi(I, J) = \Delta(I, J)$ for all ideals \( I \subseteq V \) and \( J \subseteq B \). Then the map
\[
\Phi \colon \operatorname{Min\text{-}Primal}(V) \times \operatorname{Min\text{-}Primal}(B) \to \operatorname{Min\text{-}Primal}(V \otimes^{\mathrm{tmin}} B)
\]
is a homeomorphism.
\end{theorem}

\begin{proof}
Let 
\[
\Phi' \colon \operatorname{Min\text{-}Primal}(\mathcal{A}(V)) \times \operatorname{Min\text{-}Primal}(B) \to \operatorname{Min\text{-}Primal}(\mathcal{A}(V) \otimes^{\mathrm{min}} B)
\]
and
\[
\Delta' \colon \operatorname{Min\text{-}Primal}(\mathcal{A}(V)) \times \operatorname{Min\text{-}Primal}(B) \to \operatorname{Min\text{-}Primal}(\mathcal{A}(V) \otimes^{\mathrm{min}} B)
\]
be the maps defined in the same way as \( \Phi \) and \( \Delta \).

Since \( \Phi(I, J) = \Delta(I, J) \) for all  ideals \( I \subseteq V \), \( J \subseteq B \), applying the functor \( \mathcal{A} \) to both sides yields
\[
\Phi'(\mathcal{A}(I), J) = \Delta'(\mathcal{A}(I), J)
\]
for all  ideals \( \mathcal{A}(I) \subseteq \mathcal{A}(V) \) and \( J \subseteq B \). Thus, \( \Phi' = \Delta' \), and hence \( \Phi' \) is a homeomorphism by (\cite{AR}, Theorem $4.1$).

Therefore, we obtain the following commutative diagram:
\[
\begin{tikzcd}
  \operatorname{Min\text{-}Primal}(V) \times \operatorname{Min\text{-}Primal}(B)
    \arrow[r, "\Phi"]
    \arrow[d, swap, "\theta \times 1"]
  &
  \operatorname{Min\text{-}Primal}(V \otimes^{\text{tmin}} B)
    \arrow[d, swap, "\Theta'"]
  \\
  \operatorname{Min\text{-}Primal}(\mathcal{A}(V)) \times \operatorname{Min\text{-}Primal}(B) 
    \arrow[r, swap, "\Phi'"]
  &
  \operatorname{Min\text{-}Primal}(\mathcal{A}(V)\otimes^{\mathrm{min}} B)
\end{tikzcd}
\]

Since \( \theta \times 1 \), \( \Theta' \), and \( \Phi' \) are all homeomorphisms and the diagram commutes, it follows that \( \Phi \) is also a homeomorphism.
\end{proof}

From Example $4.1$, we know that for $x_0 \in X$ and a primal ideal $P$ of $V$,  \( I_{x_0, P} = \{ f \in C(X, V) \mid f(x_0) \in P \} \) is a primal ideal of \( C(X, V) \). Let \( I \) be an arbitrary primal ideal of \( C(X, V) \). Then  \(\mathcal{A}(I)\) is a primal ideal of \( C(X, \mathcal{A}(V)) \) (see Remark 4.1). By  (\cite{za}, Proposition 4.2) we have 
\[
\mathcal{A}(I) = I_{x_0, \mathcal{A}(P)} = \mathcal{A}(I_{x_0, P})
\]
for some primal ideal \( P \) of \( V \). It follows that \( I = I_{x_0, P} \). Thus, we obtain the following result:

\begin{proposition}
    Suppose that \( I \) is a primal ideal of \( C(X, V) \). Then there exist \( x_0 \in X \) and a primal ideal \( P \) of \( V \) such that 
    \[
    I = I_{x_0, P}.
    \]
\end{proposition}

\begin{example}
Let $\mathcal{H}$ be an infinite dimensional separable Hilbert spaces and  \( V \) be an exact TRO. Then \( C(X, V) \) is also exact, since its linking \( C^* \)-algebra \( C(X, \mathcal{A}(V)) \) is exact. By Propositions 4.1 and 4.4, it follows that
\[
I_{x_0, P} \otimes^{\text{tmin}} B(\mathcal{H}) + C(X, V) \otimes^{\text{tmin}} \mathcal{K}(\mathcal{H})
\]
is a primal ideal of \( C(X, V) \).
\end{example}

\end{document}